\documentclass[10pt]{amsart}
\usepackage{amsmath}
\usepackage{amssymb,amscd}
\usepackage{graphicx}
\usepackage{mathdots}
\usepackage{mathrsfs} 
\usepackage{graphicx}
\usepackage{bbm} 
\usepackage{microtype} 

\usepackage{colonequals} 
\usepackage{mathtools}

\usepackage{tikz,tkz-euclide}
\usepackage{tikz-cd}
\usetikzlibrary{arrows}
\usepackage{tikz-3dplot}

\usepackage{color}\definecolor{darkblue}{rgb}{0,0.1,.5}
\usepackage[colorlinks=true,linkcolor=darkblue, urlcolor=darkblue, citecolor=darkblue]{hyperref}
\usepackage{cleveref}

\theoremstyle{plain}
\newtheorem{theorem}{Theorem}[section]

\newtheorem{proposition}[theorem]{Proposition}
\newtheorem{corollary}[theorem]{Corollary}

\theoremstyle{definition}
\newtheorem{definition}[theorem]{Definition}
\newtheorem{example}[theorem]{Example}

\newtheorem{construction}[theorem]{Construction}

\theoremstyle{remark}
\newtheorem*{remark}{Remark}

\numberwithin{equation}{section}

\def \begineq{\begin{equation}}
\def \endeq{\end{equation}}

\def \bb{\mathbb}

\def \Z{{\bb{Z}}}
\def \Q{{\bb{Q}}}
\def \R{{\bb{R}}}
\def \C{{\bb{C}}}

\def\Bier{\mathrm{Bier}}

\def\M{\mathrm{max}}
\def\MF{\mathrm{min}}

\def \Conv{{\rm Conv}}

\newcommand{\link}{\mathrm{link}}
\newcommand{\cat}{\mathrm{cat}}
\newcommand{\cp}{\mathrm{cup}}

\def\zk{\mathcal Z_K}
\def\rk{\mathcal R_K}

\DeclareMathAlphabet{\mathbbmsl}{U}{bbm}{m}{sl}

\makeatletter
\@namedef{subjclassname@2020}{%
  \textup{2020} Mathematics Subject Classification}
\makeatother

\title[On a class of toric manifolds arising from simplicial complexes]{On a class of toric manifolds arising from simplicial complexes}

\author[Limonchenko]{Ivan Limonchenko}
\address{Mathematical Institute of the Serbian Academy of Sciences
and Arts (SASA), Belgrade, Serbia}
\email{ivan.limoncenko@turing.mi.sanu.ac.rs}

\author[Timotijevi\'c]{Marinko Timotijevi\'c}
\address{Faculty of Science, University of Kragujevac, Serbia}
\email{marinko.timotijevic@pmf.kg.ac.rs}

\author[\v{Z}ivaljevi\'c]{Rade \v{Z}ivaljevi\'c}
\address{Mathematical Institute of the Serbian Academy of Sciences
and Arts (SASA), Belgrade, Serbia}
\email{rade@turing.mi.sanu.ac.rs}



\begin{document}

\begin{abstract}
Given an arbitrary abstract simplicial complex $K$ on $[m]:=\{1,2,\ldots, m\}$, different from the simplex $\Delta_{[m]}$ with $m$ vertices, we introduce and study a canonical $(2m-2)$-dimensional toric manifold $X_K$, associated to the canonical $(m-1)$-dimensional complete regular fan $\Sigma_K$.
This construction yields an infinite family of toric manifolds that are not quasitoric and provides a topological proof of the Dehn-Sommerville relations for the associated Bier sphere $\Bier(K)$. Finally, we classify the canonical real and complex moment-angle manifolds of Lusternik-Schnirelmann category $\leq 2$ and prove a criterion for orientability of the canonical real toric manifolds.
\end{abstract}

\maketitle

\section{Introduction}

The connection between topological properties of polyhedral products~\cite{BBCG} of a simplicial complex $K$ on $[m]$:
$$
(\underline{X},\underline{A})^K=\mathop{\bigcup}_{I\in K}\{(x_1,\ldots,x_m)\,|\,x_i\in A_i\text{ if }i\notin I\}\subseteq\prod\limits_{i=1}^{m}X_i,
\text{ where }
\underline{X}=\{X_i\,|\,1\leq i\leq m\},\, \underline{A}=\{A_i\,|\,1\leq i\leq m\},
$$
and $(X_i,A_i)$ is a topological pair for each $1\leq i\leq m$, and combinatorial properties of $K$ is one of the key problems considered in toric topology~\cite{TT}. In the particularly important and well-studied cases of a moment-angle complex $\zk=(\mathbb{D}^2,\mathbb{S}^1)^K$ and its real analogue $\rk=(\mathbb{D}^1,\mathbb{S}^0)^K$, see~\cite[Chapter 4]{TT}, it turns out that these polyhedral products are not merely cellular spaces, but also acquire a structure of a topological, or even a smooth manifold consistent with the coordinate torus action, provided $K$ is a sphere, see Definition~\ref{SphereTypesDef} and Theorem~\ref{BasicToricPropertiesTheorem}.

Moreover, if the Buchstaber number $s(K)$ of a starshaped sphere $K$ (i.e. the maximal rank of a toric subgroup $H$ in the coordinate torus $\mathbb{T}^m$ acting freely on $\zk$) is maximal possible (i.e. equals $m-n$, where $n=\dim K+1$), then the topological toric manifold $\zk/H$ is defined. This smooth $2n$-dimensional manifold with an $n$-dimensional compact torus $T^m/H$ action is a topological counterpart of a toric manifold: a complete nonsingular toric variety $X_\Sigma$ of complex dimension $n$, which is defined in the framework of toric geometry~\cite{CLS} in terms of a complete regular fan $\Sigma$ in $\R^n$, see Definition~\ref{FanDef}. The real analogues of the Buchstaber invariant and the (topological) toric manifold are defined similarly.

If the sphere $K$ is not merely starshaped, but is also polytopal, the corresponding topological toric manifold becomes a quasitoric manifold $M(P,\Lambda)$, see~\cite{DJ}, where $P$ is a combinatorial $n$-dimensional convex simple polytope with $m$ facets such that $K\cong\partial P^*$ and $H=\ker(\exp\Lambda\colon \mathbb{T}^m\to\mathbb{T}^n)$ is determined by a characteristic map $\Lambda\colon\Z^m\to\Z^n$, for which the images of the basis vectors corresponding to facets of $P$ that meet at a vertex of $P$ form a lattice basis in $\Z^n$. Furthermore, the orbit space of the $n$-dimensional compact torus $\mathbb{T}^m/H$ on $M(P,\Lambda)$ is homeomorphic as a manifold with corners to the polytope $P$. Its algebraic counterpart is a projective toric manifold $X_P$, provided the polytope $P$ is Delzant; that is, $P$ has a realization as a lattice simple polytope in $\R^n$ such that its normal fan $\Sigma_P$ is regular, see Example~\ref{NormalFanExample}. Again, the real analogues of these two constructions can be defined in a similar way.

This picture shows that each class of combinatorial spheres is of interest in toric geometry and toric topology, since it can potentially provide us with a new class of smooth manifolds (nonsingular algebraic varieties) with a compact (algebraic) torus action, whose orbit spaces carry rich combinatorial structures. However, in order to make use of a class of combinatorial spheres following the constructions of toric geometry and topology described above, we need to check starshapedness and polytopality of those spheres and then compute their real and complex Buchstaber numbers. Bier spheres will play the role of such a class of combinatorial spheres in this paper.

By definition, first introduced in the unpublished note~\cite{Bier}, given an arbitrary simplicial complex $K\neq\Delta_{[m]}$ on $[m]$ with $m\geq 2$, its Bier sphere $\Bier(K)$ equals the deleted join of the complex $K$ and its Alexander dual complex $K^\vee$, see Definition~\ref{BierDef}. Different proofs of the fact that this construction yields a PL-sphere of dimension $(m-2)$ appeared in~\cite{Matousek} and~\cite{dL}. More recently, a canonical starshaped realization for any Bier sphere was obtained in~\cite{Zivaljevic19}, see also \cite{Zivaljevic21} and Theorem~\ref{BierSpheresAreStarshapedTheorem}. On the other hand, it was observed in~\cite{Matousek} that almost all Bier spheres of simplicial complexes on $[m]$ are non-polytopal, as $m\to\infty$; however, no particular example of a non-polytopal Bier sphere has been constructed so far.

The first application of Bier spheres in toric topology was given in~\cite{DS}, where an example of a moment-angle complex $\zk$ was constructed such that $\zk$ is a topological manifold and there exists a non-trivial indecomposable triple Massey product in the cohomology ring $H^*(\zk)$. A systematic study of Bier spheres from the viewpoint of toric geometry and topology was initiated in~\cite{LS}, where a general formula for Buchstaber invariant $s(\Bier(K))$ was proven and the corresponding characteristic map was introduced (both under the assumption that the ghost vertices of $K$ are taken into account). These results were generalized in~\cite{LV}: it turned out that Bier spheres are starshaped spheres having maximal possible Buchstaber numbers, both in the real and complex cases. On the other hand, methods of toric topology yielded the classification of all minimally non-Golod Bier spheres~\cite{LZ} and allowed for an explicit description of
the cohomology of real toric spaces associated with Bier spheres~\cite{CYY}.

This paper is an expansion of a short expository note~\cite{LTZ} in which some of the present paper results, namely, Theorem~\ref{mainthm}, Proposition~\ref{DSrelationsProp}, and Proposition~\ref{ToricNonQuasitoricProp}, were announced without proofs. We begin with a canonical regular realization for any Bier sphere and then discuss some applications of our construction in toric topology, toric geometry and geometrical combinatorics. In particular, this allows us to generalize (refine) the results on the canonical starshaped realization \cite{Zivaljevic19, Zivaljevic21} and the Buchstaber numbers computation for Bier spheres \cite{LS, LV}.

The structure of the paper is as follows. In Section 2, we collect all the necessary definitions, constructions, and results related to simplicial complexes and fans. Based on that, we introduce the notion of a Bier sphere and recall its basic combinatorial and geometric properties. The central result in Section 3 is Theorem~\ref{mainthm}. Here we focus on the regularity of the canonical, complete fan $\Sigma_K$ (the radial fan of the starshaped realization of $\Bier(K)$, described in Theorem~\ref{BierSpheresAreStarshapedTheorem}). As a first application we observe that whenever $\Bier(K)$ is polytopal the combinatorial  simple polytope $P$ such that $\Bier(K)\cong\partial P^*$ is Delzant. In Section 4, we introduce the canonical (real) moment-angle manifolds and (real) toric manifolds corresponding to the canonical fan $\Sigma_K$ and classify all (real) canonical moment-angle manifolds of Lusternik-Schnirelmann category not greater than two (Theorem~\ref{LScatBierTh}). As the first consequence of our main construction, we obtain a topological proof of the Dehn-Sommerville relations for Bier spheres based on Poincar\'e duality for toric manifolds (Proposition~\ref{DSrelationsProp}). Then an infinite family of toric manifolds that are not quasitoric is identified taking into account the canonical toric manifolds over non-polytopal Bier spheres (Proposition~\ref{ToricNonQuasitoricProp}). Finally, a criterion for orientability of the canonical real toric manifolds over Bier spheres is proved based on the canonical characteristic matrix $\Lambda_K$ construction (Proposition~\ref{OrientabilityProp}).

\subsection*{Acknowledgements}
The authors are grateful to Anton Ayzenberg, Li Cai, Taras Panov, and Ale\v{s} Vavpeti\v{c} for various fruitful discussions, valuable comments and suggestions. Limonchenko and \v{Z}ivaljevi\'c were supported by the Serbian Ministry of Science, Technological Development and Innovation through the Mathematical Institute of the Serbian Academy of Sciences and Arts.

\section{Abstract simplicial complexes and rational polyhedral fans}

In this section, we discuss some of the basic constructions related to simplicial complexes and fans. We also introduce the main object of our study, a Bier sphere, and its basic combinatorial and geometric properties.

\begin{definition}
Set $[m]:=\{1,2,\ldots,m\}$ for an integer $m\geq 1$. We say that a non-empty subset $K\subseteq 2^{[m]}$ is an {\emph{(abstract) simplicial complex on}} $[m]$ if
$\sigma\in K$ and $\tau\subseteq\sigma$ imply
$\tau\in K$.
Elements of $K$ are called its {\emph{simplices}}, or {\emph{faces}}. Faces of cardinality $1$ are called {\emph{(geometric) vertices}}. The set of vertices of $K$ will be denoted by $V(K)$.

Maximal (by inclusion) simplices of $K$ are called its {\emph{facets}}, and their set is denoted by $\M(K)$; we say that facets of $K$ \emph{generate} it and write $K=\langle \M(K)\rangle$.

Minimal (by inclusion) subsets in $[m]$ that are not simplices of $K$ are called its {\emph{minimal non-faces}}, and their set is denoted by $\MF(K)$. Minimal non-faces of cardinality $1$ are called {\emph{ghost vertices}}.

Recall that the {\emph{dimension}} of a simplex is one less than its cardinality; the {\emph{dimension}} of a simplicial complex $K$ equals the maximal dimension of its simplex and is denoted by $\dim(K)$. If all maximal faces of $K$ have the same dimension, then $K$ is called {\emph{pure}}.
\end{definition}

Given two simplicial complexes, $K_1$ and $K_2$, on $[m]$, we say that they are \emph{isomorphic} and write $K_1\cong K_2$, if there exists a bijection $\phi\colon [m]\to [m]$ such that both $\phi$ and its inverse $\phi^{-1}$ map simplices to simplices. For an $(n-1)$-dimensional simplicial complex $K$, we define its \emph{$f$-vector} to be the $(n+1)$-tuple $(f_{-1},f_0,\ldots,f_{n-1})$ such that $f_i$ equals the number of $i$-dimensional faces of $K$ for $i\geq -1$. Note that $f_{-1}=1$. Then the \emph{$h$-vector} of $K$ is defined to be the $(n+1)$-tuple $(h_0,h_1,\ldots,h_n)$ determined from the identity:
$$
\sum\limits_{i=0}^n h_it^{n-i}=\sum\limits_{j=-1}^{n-1}f_{j}(t-1)^{n-j-1}.
$$

\begin{example}
Consider the next three simplicial complexes on $[m]$ with $m\geq 1$:
\begin{itemize}
\item $\varnothing_{[m]}:=\{\varnothing\}$, the void complex;
\item $\Delta_{[m]}:=2^{[m]}$, the simplex;
\item $\partial \Delta_{[m]}:=2^{[m]}\setminus\{[m]\}$, the boundary of the simplex.
\end{itemize}
Note that $\dim(\varnothing_{[m]})=-1$, $\dim(\Delta_{[m]})=m-1$, and $\dim(\partial\Delta_{[m]})=m-2$ for any $m\geq 1$. Moreover, all the three simplicial complexes are pure and $\MF(\Delta_{[m]})=\varnothing$, $\MF(\partial\Delta_{[m]})=\{[m]\}$, while $\varnothing_{[m]}$ has $m$ ghost vertices.
\end{example}

Next, we are going to describe briefly all the classes of combinatorial spheres that we are using in this paper.

\begin{definition}\label{SphereTypesDef}
We call a pure $(n-1)$-dimensional simplicial complex $K$
\begin{itemize}
\item a \emph{polytopal sphere}, if it is isomorphic to the boundary of a convex simplicial $n$-polytope;
\item a \emph{starshaped sphere}, if there is a geometric realization of $K$ in $\R^n$ and a point $p$ in $\R^{n}$ such that each ray emanating from $p$ meets this realization in exactly one point;
\item a \emph{PL-sphere}, if there is a geometric realization of $K$ PL-homeomorphic to $\partial\Delta^{n}$;
\item a \emph{simplicial sphere}, if there is a geometric realization of $K$ homeomorphic to $S^{n-1}$;
\item a (rational) \emph{homology sphere}, if for each $\sigma\in K$ one has:
\[
\tilde{H}_{i}(\link_{K}(\sigma);\Q)=\begin{cases}
0,&\text{if $i<n-1-|\sigma|$;}\\
\Q,&\text{if $i=n-1-|\sigma|$,}
\end{cases}
\]
where $\link_K(\sigma):=\{\tau\in K\,|\,\sigma\cup\tau\in K, \sigma\cap\tau=\varnothing\}$ and $|\sigma|$ denotes the cardinality of $\sigma$.
\end{itemize}
\end{definition}

Observe that the five classes of combinatorial spheres listed above form an increasing sequence with respect to inclusion when viewed from top to bottom. Moreover, when $n=2$ these classes coincide with each other.

Now let us introduce the main object of study in this paper, a Bier sphere. Consider a copy $[m']:=\{1',2',\ldots,m'\}$ of the set $[m]$ with $m\geq 2$.

\begin{definition}\label{BierDef}
Suppose $K\neq\Delta_{[m]}$ is a simplicial complex on $[m]$ with $m\geq 2$. Then
\begin{itemize}
\item the {\emph{Alexander dual}} of $K$ is a simplicial complex $K^\vee$ on $[m']$ such that
$$
I\in \MF(K)\Longleftrightarrow [m']\setminus I'\in \M(K^\vee);
$$
\item the {\emph{Bier sphere}} of $K$ is a simplicial complex $\Bier(K)$ on the disjoint union $[m]\sqcup [m']$ such that
$$
\Bier(K)=\{I\sqcup J'\,|\,I\in K, J'\in K^\vee, I\cap J=\varnothing\};
$$
that is, $\Bier(K)$ is the {\emph{deleted join}} of $K$ and $K^\vee$. Obviously, $(K^\vee)^\vee\cong K$ and hence $\Bier(K)\cong\Bier(K^\vee)$.
\end{itemize}
\end{definition}

Due to~\cite[Corollary 2.7]{LS}, the Bier sphere $\Bier(K)$ has $m-f_0(K)+f_{m-2}(K)$ ghost vertices. Hence it has  $f_0(\Bier(K))=2m-(m-f_0(K)+f_{m-2}(K))=m+f_0(K)-f_{m-2}(K)$ geometric vertices.

\begin{example}
It follows easily from the above definition that in dimension 1 the only Bier spheres are the boundaries of a $p$-gon with $p=3, 4, 5$, and $6$, see~\cite[Example 2.8]{LZ}. In dimension 2, there exist $13$ different combinatorial types of Bier spheres, see~\cite[Theorem 2.16]{LS}.
\end{example}

It was shown in~\cite{Bier} that $\Bier(K)$ defined above is a PL-sphere of dimension $(m-2)$ with the number of vertices varying between $m$ and $2m$. An explicit construction of the required PL-homeomorhism between $\Bier(K)$ and $\partial\Delta^{m-1}$ can be found in~\cite{dL}. Here is a generalization of this result, which is crucial for this paper.

\begin{theorem}[\cite{Zivaljevic19, Zivaljevic21}]\label{BierSpheresAreStarshapedTheorem}
Let $K\neq\Delta_{[m]}$ be a simplicial complex on $[m]$ with $m\geq 3$. Then $\Bier(K)$ has a geometric realization as a starshaped sphere in the hyperplane $H_0:=\{x\in\R^m\,|\,\langle u,x\rangle=0\}$, where $u$ is the sum of the standard basis vectors $e_i, 1\leq i\leq m$ in $\R^m$.
\end{theorem}

In the next section we  revisit and refine this result by putting it in the context of toric topology. We finish this section with a discussion of fans.

\begin{definition}
Consider the standard integer lattice $\Z^n\subset\R^n$. Each set of vectors $S=\{v_1,\ldots,v_s\}\subset\Z^n$ generates a \emph{(rational polyhedral) cone}
$$
\sigma=\R_{\ge}\langle S\rangle:=\{r_1v_1+\cdots+r_sv_s\,|\,r_i\geq 0\text{ for each }1\leq i\leq s\}\subseteq\R^n.
$$
In what follows, we consider only \emph{strongly convex} cones $\sigma$; that is, $\sigma$ does not contain any line in $\R^n$. A cone is called \emph{simplicial} (respectively, \emph{nonsingular}) if it is generated by a part of a basis of $\R^n$ (respectively, $\Z^n$).
\end{definition}

\begin{definition}\label{FanDef}
A \emph{(rational polyhedral) fan} in $\R^n$ is a set $\Sigma$ of cones in $\R^n$ such that each face of a cone in $\Sigma$ is again a cone of $\Sigma$ and the intersection of any two cones in $\Sigma$ is a face of each of them. A fan $\Sigma$ in $\R^n$ is called
\begin{itemize}
\item \emph{complete}, if the union of all its cones equals $\R^n$;
\item \emph{simplicial}, if all its cones are simplicial;
\item \emph{regular}, if all its cones are nonsingular.
\end{itemize}
\end{definition}

\begin{example}\label{NormalFanExample}
Suppose $P$ is a lattice simple polytope in $\R^n$. Consider primitive integer normal vectors to its facets. The fan $\Sigma_P$ whose maximal (by inclusion) cones are generated by the normal vectors to facets intersecting at a vertex of $P$ is called a \emph{normal fan} of $P$. Obviously, it is a complete simplicial fan. On the other hand, if $P$ is a lattice polytope combinatorially equivalent to the dodecahedron, then $\Sigma_P$ can not be a regular fan, since in dimension 3 a polytope with a regular normal fan must have either a triangular or a quadrangular facet, see~\cite{Delaunay}.
\end{example}

Finally, we introduce a construction that relates fans to simplicial complexes. Suppose $\Sigma$ is a simplicial fan and consider the $m$-tuple $(v_1,\ldots,v_m)$ of primitive integer generators of the 1-dimensional cones of $\Sigma$. This determines the \emph{underlying simplicial complex} $\mathcal K_\Sigma$ on $[m]$ as follows: by definition, $\{i_i,\ldots,i_p\}\in \mathcal K_\Sigma$ if and only if the vectors $v_{i_1},\ldots,v_{i_p}$ generate a cone in $\Sigma$. In particular, $\Sigma$ is a complete simplicial fan if and only if $\mathcal K_\Sigma$ is a starshaped sphere, see~\cite{Zivaljevic19}.

\section{Regular realizations for Bier spheres}
In this section, we refine Theorem~\ref{BierSpheresAreStarshapedTheorem} by showing that each Bier sphere admits a canonical complete regular fan. More explicitly, we focus on the regularity of the canonical complete fan $\Sigma_K$, the radial fan of the starshaped realization of $\Bier(K)$ described in~\cite{Zivaljevic19, Zivaljevic21}, see Theorem~\ref{BierSpheresAreStarshapedTheorem}. This construction allows us to link the combinatorics and combinatorial geometry of Bier spheres with the geometry and topology of toric spaces.

In what follows, we denote by $(e_1,\ldots,e_{m-1})$ the standard basis in $\R^{m-1}$ and set $e_{m}:=-(e_1+\cdots+e_{m-1})$.

\begin{construction}
Let $I, J\subset [m]$ and $I\cap J=\varnothing$. We consider the following sets:
$$
G(I,J):=\{-e_{i}, e_{j}\,|\,i\in I, j\in J\},\,
C(I,J):=\R_{\ge}\langle G(I,J)\rangle,\,\text{ and }\,\Sigma_K:=\{C(I,J)\,|\,I\in K, J'\in K^\vee\}.
$$
\end{construction}

\begin{example}\label{TheFirstBierSphere} The Alexander dual of a simplicial complex $\varnothing_{[m]}$ is $\partial\Delta_{[m']}$, so the associated Bier sphere is $\Bier(\varnothing_{[m]})\cong\partial\Delta_{[m]}$. Since $P=\Conv\{e_1,\ldots,e_{m-1},e_m\}$ is a simplex-polytope containing the origin in its interior ($0=e_1+\cdots+e_{m-1}+e_m$) and for every $J\in \partial\Delta_{[m]}$ the geometric simplex  $\Conv\, G(\varnothing,J)$ is a face of $P$, we conclude that $\Sigma_{\varnothing_{[m]}}$ is a complete regular fan with underlying simplicial complex $\partial\Delta_{[m]}$.
\end{example}

\begin{figure}[h]
\includegraphics[scale=0.45]{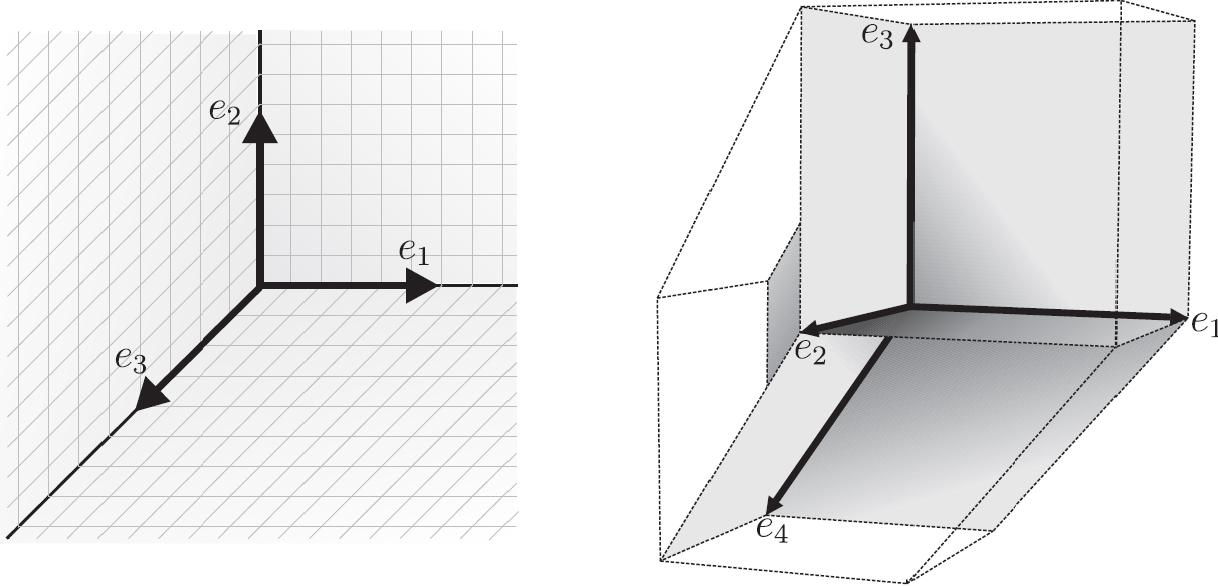}
\caption{The canonical complete regular fans $\Sigma_{\varnothing_{[3]}}$ and $\Sigma_{\varnothing_{[4]}}$.}
\end{figure}

The following theorem can be deduced from the original proofs \cite{Zivaljevic19, Zivaljevic21} of Theorem \ref{BierSpheresAreStarshapedTheorem}. However, for the reader's convenience, here we include all the details, which provide also an independent proof of  Theorem \ref{BierSpheresAreStarshapedTheorem}.

\begin{theorem}\label{mainthm}
Let $K\neq\Delta_{[m]}$ be a simplicial complex on $[m]$ with $m\geq 2$. Then $\Sigma_K$ is a complete regular fan in $\R^{m-1}$ and $\Bier(K)$ is isomorphic to the underlying simplicial complex of $\Sigma_K$.
\end{theorem}
\begin{proof}
Note that the statement is obvious for $m=2$. Therefore, assume $m\geq 3$. First, let us show that $\Sigma_K$ is a fan. By definition, $C(I,J)$ is a rational polyhedral cone generated by the set $G(I,J)$ of vectors in $\R^{m-1}$. Moreover, $G(I,J)$ is a minimal set of generators for $C(I,J)$, since $G(I,J)$ is a linearly independent set of vectors in $\R^{m-1}$.


To prove that the intersection of any two cones in $\Sigma_K$ is again a cone of $\Sigma_K$, it is sufficient to show that $\Sigma_K$ satisfies the following condition:
\begin{equation}\label{eq:prva}
\text{ for each } x\in \mathbb{R}^{m-1}\, \text{ there exists a unique }I\,\sqcup J'\in \Bier (K)\text{ such that }x\in C^\circ(I,J),
\end{equation}
where $C^\circ(I,J)$ is the "interior" of $C(I,J)\in\Sigma_K$, consisting of linear combinations of $-e_i$ and $e_j$ with strictly positive scalars. This would also imply that the resulting fan $\Sigma_K$ is complete.

Here we use a relation between Bier spheres, which are sufficiently close. It is well known (\cite{Matousek}, proof of Theorem 5.6.2) that if $K$ and $K\cup \{A\}$ are simplicial complexes with vertices in $[m]$, then:
$$
\Bier(K\cup A)=\Bier(K)\setminus D_1\cup D_2,
$$
where
$$
D_1=\{(A\setminus\{i\})\sqcup A^c\mid i\in A\},$$
$$D_2=\{A\sqcup (A^c \setminus\{i\}) \mid i\in A^c\}.
$$

Similarly, there is an induced relation for our construction:
$$
\Sigma_{K\cup \{A\}}=\Sigma_K\setminus\{C(I,J)\mid I\sqcup J\in D_1\}\cup\{C(I,J)\mid I\sqcup J\in D_2\}.
$$

Thus, in order to prove that $\Sigma_K$ satisfies \eqref{eq:prva} for all complexes $K\subset 2^{[m]}$, it is sufficient to show that
\begin{itemize}
\item[$\bullet$] $\Sigma_{\varnothing_{[m]}}$ satisfies \eqref{eq:prva}  (which is shown in Example \ref{TheFirstBierSphere}, because $\Sigma_{\varnothing_{[m]}}$ is a complete regular fan);
\item[$\bullet$]if $\Sigma_K$ satisfies \eqref{eq:prva} and $A\subseteq[m]$ is a minimal non-face of $K$, then $\Sigma_{K\cup\{A\}}$ also satisfies \eqref{eq:prva}.
\end{itemize}

Suppose $\Sigma_K$ satisfies \eqref{eq:prva} and $K\cup \{A\}$ is a simplicial complex. To prove that $\Sigma_{K\cup\{A\}}$ satisfies \eqref{eq:prva}, it is sufficient to show that all the points of 
$$
\mathbb{D}_2=\bigcup\limits_{I\sqcup J\in D_2} C(I,J)
$$
are contained in the interior of the unique cone of $\{C(I,J)\mid I\sqcup J\subset A\sqcup(A^c\setminus \{i\})\in D_2\}$.

Let $x\in\mathbb{D}_2$ be arbitrary. Since $\Sigma_K$ satisfies \eqref{eq:prva}, the point $x$ has to be contained in $C(A\setminus\{i\},A^c)$ for some $i\in A$ (this can be shown as in the argument that follows) and $x\not\in C(I,J)$ for any $I\sqcup J\in \Bier (K)\setminus D_1$, since otherwise we would get a contradiction with \eqref{eq:prva}. Therefore,
\begin{equation}\label{eq:razlaganje}
x=\alpha_{i_1} e_{i_1}+\cdots+\alpha_{i_k} e_{i_k}+\beta_{j_1}e_{j_1}+\cdots+\beta_{j_l}e_{j_l},
\end{equation}
where $\alpha_{i_1}\leq 0,\ldots,\alpha_{i_k}\leq 0$, $\{i_1,\ldots,i_k\}= I=A\setminus \{i\}$ and $\beta_{j_1}\geq 0,\ldots,\beta_{j_l}\geq 0$, $\{j_1,\ldots,j_l\}=J=[m]\setminus A$.

Without loss of generality, we may assume that
$$
0\leq\beta_{j_1}\leq\beta_{j_2}\leq\cdots\leq\beta_{j_l}.
$$

Here we distinguish three different cases: $m=i$, $m\in I$ and $m\in J$.

If $m\not\in I\cup J$ ($m=i$), relation \eqref{eq:razlaganje} can be rewritten as
\begin{equation}\label{eq:raz1}
x=(\alpha_{i_1}-\beta_{j_1}) e_{i_1}+\cdots+(\alpha_{i_k}-\beta_{j_1}) e_{i_k}+(\beta_{j_2}-\beta_{j_1})e_{j_2}+\cdots+(\beta_{j_l}-\beta_{j_1})e_{j_l}-\beta_{j_1}e_m
\end{equation}
and here  $\alpha_{i_1}-\beta_{j_1}\leq 0,\ldots \alpha_{i_k}-\beta_{j_1}\leq 0, -\beta_{j_1}\leq 0$ and $(\beta_{j_2}-\beta_{j_1})\geq 0, \beta_{j_l}-\beta_{j_1}\geq 0$
implying that $C(A,A^c\setminus \{j_1\})$ is the cone of $\mathbb{D}_2$ containing $x$. Because of the uniqueness of \eqref{eq:raz1}, the point $x$ will be contained in two cones $C(A,A^c\setminus \{j_1\})$ and $C(A,A^c\setminus \{j_2\})$ if and only if $\beta_{j_1}=\beta_{j_2}$ so by disregarding the coefficients which are equal to zero in \eqref{eq:raz1}, we obtain the unique cone of $\Sigma_{K\cup\{A\}}$ containing $x$ in its interior.

If $m\in I$, say $i_k=m$ and $i=i_1$, the equation \eqref{eq:raz1} becomes
$$x=\alpha_{i_2} e_{i_2}+\cdots+\alpha_{i_{k-1}} e_{i_{k-1}}+\alpha_m(-e_1-\cdots-e_{m-1})+\beta_{j_1}e_{j_1}+\cdots+\beta_{j_l}e_{j_l}$$
and by adding and subtracting $\beta_{j_1}(e_1+\cdots+e_{j_1-1}+e_{j_1+1}+\cdots+ e_{m-1})$ we obtain
\begin{align}
\nonumber x=& -\beta_{j_1}e_{i_1}+(\alpha_{i_2}-\beta_{j_1})e_{i_2}+\cdots+(\alpha_{i_{k-1}}-\beta_{j_1}) e_{i_{k-1}}+(\alpha_m-\beta_{j_1})e_m+\\
&+(\beta_{j_2}-\beta_{j_1})e_{j_2}+\cdots+(\beta_{j_l}-\beta_{j_1})e_{j_l}.\label{eq:raz2}
\end{align}
Since $-\beta_{j_1}\leq 0,\alpha_{i_2}-\beta_{j_1}\leq 0,\ldots,\alpha_{i_{k-1}}-\beta_{j_1}\leq 0,\alpha_m-\beta_{j_1}\leq 0$ and $(\beta_{j_2}-\beta_{j_1})\geq 0,\ldots, \beta_{j_l}-\beta_{j_1}\geq 0$, we conclude that $C(A,A^c\setminus \{j_1\})$ is the cone of $\mathbb{D}_2$ containing $x$. The uniqueness of the cone of $\Sigma_{K\cup\{A\}}$ containing $x$ in its interior follows as in the previous case.

If $m\in J$, say $j_p=m$ and $i=i_1$, we consider two cases. First, if $\beta_{i_1}<\beta_{m}$ the equation \eqref{eq:razlaganje} becomes
$$x=\alpha_{i_2} e_{i_2}+\alpha_{i_k} e_{i_k}+\beta_{j_1}e_{j_1}+\cdots+\beta_m(-e_1-\cdots-e_{m-1})+\cdots +\beta_{j_l}e_{j_l}$$
and with a similar modification as in the previous cases we obtain
\begin{align}
\nonumber x=& -\beta_{j_1}e_{i_1}+(\alpha_{i_2}-\beta_{j_1})e_{i_2}+\cdots+(\alpha_{i_k}-\beta_{j_1}) e_{i_k}\\
&+(\beta_{j_2}-\beta_{j_1})e_{j_2}+\cdots+(\beta_m-\beta_{j_1})e_m+ \cdots+(\beta_{j_l}-\beta_{j_1})e_{j_l}.\label{eq:raz3}
\end{align}
proving that $x$ belongs to the cone $C(A,A^c\setminus \{j_1\})$ and the corresponding interior of the unique sub-cone. \medskip

Second, if $\beta_m=\beta_{j_1}$ the equation \eqref{eq:razlaganje} becomes
\begin{align*}
x&=\alpha_{i_2} e_{i_2}+\cdots+\alpha_{i_k} e_{i_k}+\beta_m(-e_1-\cdots-e_{m-1})+\beta_{j_2}e_{j_3}+\cdots+\beta_{j_l}e_{j_l}\\
&=-\beta_m e_{i_1}+(\alpha_{i_2}-\beta_{m})e_{i_2}+\cdots+(\alpha_{i_k}-\beta_{m}) e_{i_k}
+(\beta_{j_2}-\beta_{m})e_{j_2}+\cdots+(\beta_{j_l}-\beta_{m})e_{j_l}
\end{align*}
proving that the cone $C(A,A^c\setminus \{m\})$, and therefore its corresponding sub-cone, contains $x$ in its interior.


Finally, let us show that the complete fan $\Sigma_K$ is regular. It suffices to observe that the minimal set $G(I,J)$ of generators for a cone $C(I,J)$ is a (part of) basis of the lattice $\Z^{m-1}\subset\R^{m-1}$ for each pair $I\in K$ and $J'\in K^\vee$ such that $I\cap J=\varnothing$. This finishes the proof.
\end{proof}

\begin{remark}
Note that, given that $\Sigma_K$ is a fan and that  $\Bier(K)\cong\mathcal K_{\Sigma_K}$ is a simplicial sphere, we can conclude that the fan $\Sigma_K$ is complete.
\end{remark}

Recall that a (combinatorial) simple polytope is a \emph{Delzant polytope} if it has a realization as a lattice polytope whose normal fan is regular. Several important classes of (convex) simple polytopes, including nestohedra~\cite{Z} and 2-truncated cubes~\cite{BV}, were proven to consist of Delzant polytopes. As an immediate corollary of the previous theorem, we obtain the following statement.

\begin{corollary}
For any polytopal Bier sphere, the corresponding convex simple polytope is Delzant.
\end{corollary}

\begin{example}
Let $m=3$ and $\M(K)=\{\{1\}, \{2,3\}\}$. Then
$$
\Bier(K)=\langle\{1,2'\}, \{1,3'\}, \{2,3\}, \{2,3'\}, \{3,2'\}\rangle,
$$
and we obtain the complete regular fan $\Sigma_K$ drawn below.
\end{example}

\begin{figure}[h]
\includegraphics[scale=0.4]{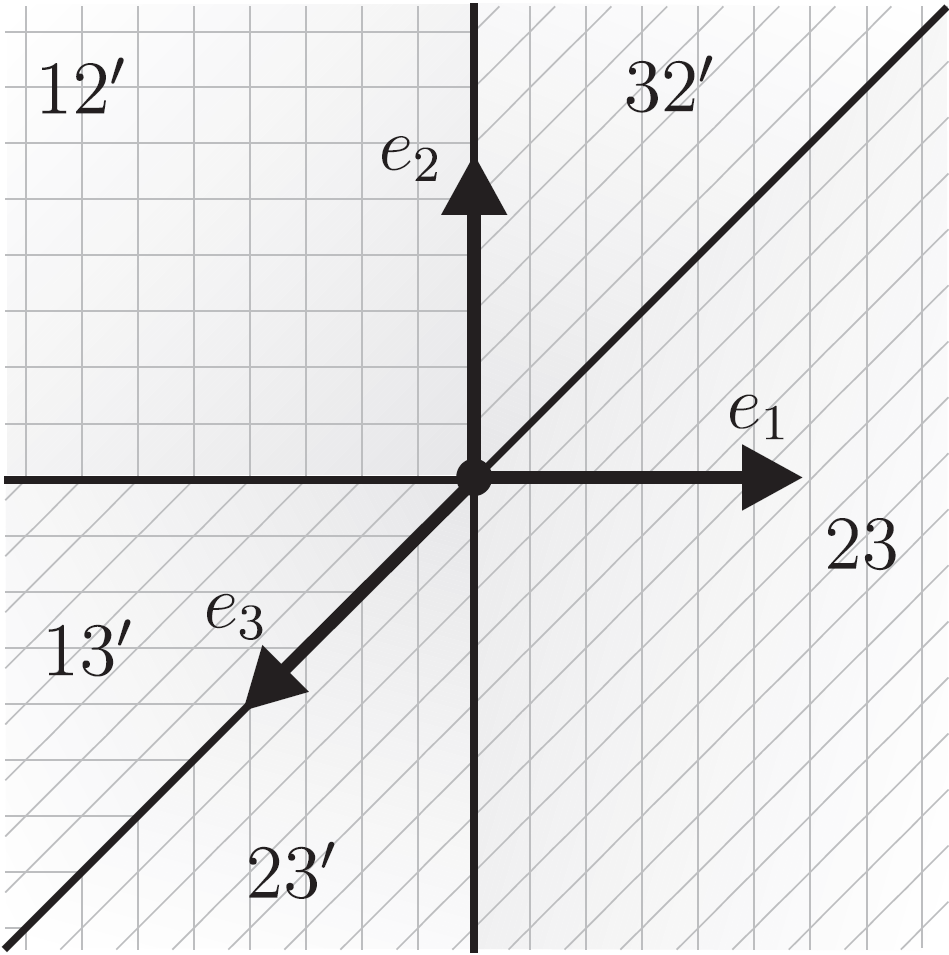}
\caption{The canonical complete regular fan $\Sigma_K$ for the complex $K=\langle\{1\},\{2,3\}\rangle$.}
\end{figure}\label{CanFanFor3VerFig}

\section{Applications}

In this section, we make use of our construction of the canonical complete regular fan $\Sigma_K$ in order to obtain some applications of Theorem~\ref{mainthm} in topology and combinatorics. We start with toric varieties, the key notion arising in toric geometry~\cite{CLS}. By definition, a \emph{toric variety} $X$ of (complex) dimension $n$ is a normal complex algebraic variety containing the algebraic torus $(\mathbb C^*)^n$ as a Zariski open subset in such a way that the natural action of $(\mathbb C^*)^n$ on itself extends to the action of $(\mathbb C^*)^n$ on $X$.

A classical result of toric geometry is the fundamental correspondence between toric varieties $X$ of complex dimension $n$ and rational polyhedral fans $\Sigma$ in $\R^n$, see~\cite{Danilov}, which allows us to write $X_\Sigma$ for each toric variety under our consideration. Moreover, due to this correspondence, a toric variety $X_\Sigma$ is a complete variety if and only if the fan $\Sigma$ is complete, it is a nonsingular variety if and only if the fan $\Sigma$ is regular, and it is a projective variety if and only if $\Sigma$ is a normal fan of a lattice polytope.

\begin{definition}
A \emph{toric manifold} is a complete nonsingular toric variety.
\end{definition}

Therefore, $X_\Sigma$ is a toric manifold if and only if $\Sigma$ is a complete regular fan. Applying Theorem~\ref{mainthm}, we immediately obtain the following definition. 

\begin{definition}
Let $K\neq\Delta_{[m]}$ be a simplicial complex on $[m]$ with $m\geq 2$. We call $X_K:=X_{\Sigma_K}$ the \emph{canonical (algebraic) toric manifold} over $K$.    
\end{definition}

This definition yields an infinite family of toric manifolds, with a finite number of elements in each positive even (real) dimension. Here is the first application of this notion.
In~\cite{Stanley}, the Dehn-Sommerville relations for the $h$-vector of a simplicial complex were proven to take place for each homology sphere using the theory of Stanley-Reisner rings. Our construction provides a new topological proof of this result in the case of Bier spheres; it was first obtained in~\cite{BPSZ05} using purely combinatorial technique.

\begin{proposition}\label{DSrelationsProp}
Let $K\neq\Delta_{[n+1]}$ be a simplicial complex on $[n+1]$ with $n\geq 1$. Then the Dehn-Sommerville equations $h_{i}(\Bier(K))=h_{n-i}(\Bier(K))$ hold for all $0\leq i\leq n$, where $\dim\Bier(K)=n-1$.
\end{proposition}
\begin{proof}
Due to~\cite{Danilov}, the Betti number $\beta_{2i}(X_K)$ equals $h_i(\Bier(K))$ for each $0\leq i\leq n$. It remains to apply Poincar\'e duality for the closed oriented $2n$-dimensional smooth manifold $X_K$.
\end{proof}

In toric topology, the following topological analogue of a toric manifold was introduced~\cite{DJ}.

\begin{definition} Let $P$ be a combinatorial simple polytope of dimension $n$. A \emph{quasitoric manifold} over $P$ is a smooth $2n$-dimensional manifold $M$ with a smooth action of the compact torus $\mathbb T^n$ satisfying the two conditions:
\begin{itemize}
\item[(1)] the action is locally standard;
\item[(2)] there is a continuous projection $\pi\colon\,M\rightarrow P$ whose fibers are $\mathbb T^n$-orbits.
\end{itemize}
\end{definition}

Observe that a nonsingular projective toric variety $X_\Sigma$ corresponding to the normal fan $\Sigma=\Sigma_P$ of an $n$-dimensional Delzant polytope $P$ is a quasitoric manifold of real dimension $2n$ over the combinatorial polytope $P$ via restricting the action of $(\mathbb C^*)^n$ to $\mathbb T^n\subset (\mathbb C^*)^n$. In what follows, we consider (topological) toric manifolds with this restricted action and we denote them by $M_\Sigma$. The existence of the class of quasitoric manifolds arising as toric manifolds over Delzant polytopes motivated the problem, see~\cite[Problem 6.26]{BP04}, of constructing a toric manifold, which is not a quasitoric manifold. In~\cite{Suyama}, it was done by a consecutive resolution of singularities for a toric variety over the Barnette nonpolytopal 3-sphere with 8 vertices. The canonical complete regular fan $\Sigma_K$ construction yields a different approach to this problem.

\begin{proposition}\label{ToricNonQuasitoricProp}
There are infinitely many canonical toric manifolds whose orbit spaces are not homeomorphic to any simple polytope as manifolds with corners.
\end{proposition}
\begin{proof}
Note that, given a regular fan $\Sigma$, the toric variety $X_\Sigma$ is a projective toric manifold if and only if $\Sigma=\Sigma_P$ for a certain Delzant polytope $P$ and $M_\Sigma$ is a quasitoric manifold if and only if the underlying simplicial complex $\mathcal K_\Sigma$ is a polytopal sphere. Due to~\cite{Matousek}, almost all Bier spheres $\Bier(K)=\mathcal K_{\Sigma_K}$ for simplicial complexes $K\neq\Delta_{[m]}$ on $[m]$ are nonpolytopal, as $m\to\infty$. This implies that the corresponding canonical toric manifolds $M_K:=M_{\Sigma_K}$ are not quasitoric, hence the equivalent condition on orbit spaces in the statement follows.
\end{proof}

Now let us consider polyhedral products over Bier spheres and their partial quotients. In what follows, let $K$ be a simplicial complex on $[m]$ with $\dim K=n-1, n\geq 3$. The next construction defines moment-angle complexes, a key object of study in toric topology~\cite{TT} and a particularly important class of polyhedral products~\cite{BBCG}.

\begin{definition}
By the \emph{(complex) moment-angle complex} of $K$ we mean a topological space
$$
\zk:=\mathop{\bigcup}_{I\in K}(\mathbb D^2,\mathbb S^1)^I\subset (\mathbb D^2)^m\subset \C^m,\text{ where }(\mathbb D^2,\mathbb S^1)^I := \prod\limits_{i=1}^m\,X_i,
$$
each space $X_i=\mathbb D^2$, the unit disc in $\C$, if $i\in I$ and $X_i=\mathbb S^1$, the unit circle in $\C$, if $i\notin I$. The \emph{real moment-angle complex} $\rk$ is defined similarly, replacing the pair $(\mathbb D^2,\mathbb S^1)$ by the pair $(\mathbb D^1,\mathbb S^0)$.
\end{definition}

Note that the compact torus $\mathbb T^m$ acts coordinatewisely on $\zk$ and similarly the real torus $\Z_2^m$ acts on $\rk$. The next statement highlights the relationship between equivariant manifold structures on moment-angle complexes and combinatorial properties of the underlying simplicial complexes, which is explored in the framework of toric topology and is important for us here.

\begin{theorem}[\cite{Cai},\cite{PU,Tam}]\label{BasicToricPropertiesTheorem}
The two following criteria and a sufficient condition hold.
\begin{itemize}
\item $\rk$ is a topological $n$-manifold if and only if $K$ is a simply connected homology sphere;
\item $\zk$ is a topological $(n+m)$-manifold if and only if $K$ is a homology sphere;
\item If $K$ is a starshaped sphere, then $\rk$ and $\zk$ both acquire equivariant smooth structures.
\end{itemize}
\end{theorem}

\begin{definition}
Let $K\neq\Delta_{[m]}$ be a simplicial complex on $[m]$ with $m\geq 2$. The canonical complete regular fan $\Sigma_K$ gives rise to: 
\begin{itemize}
\item the smooth $(2m+f_0(K)-f_{m-2}(K)-1)$-dimensional \emph{canonical moment-angle manifold} $\mathcal Z_{\Bier(K)}$;
\item the smooth $(m-1)$-dimensional \emph{canonical real moment-angle manifold} $\mathcal R_{\Bier(K)}$.
\end{itemize}
\end{definition}

\begin{remark}
In the above definition, we make use of Theorem~\ref{BasicToricPropertiesTheorem} when $m\geq 4$. If $m\leq 3$, then $\Bier(K)$ is a polytopal sphere, and therefore its (real) moment-angle complex acquires an equivariant smooth structure due to~\cite{BPR}.
\end{remark}

Since $K$ is a full subcomplex in $\Bier(K)$, the previous definition and~\cite[Theorem 4.5.8]{TT} yield a smooth canonical moment-angle manifold with torsion in integer cohomology whenever $H^*(K)$ contains torsion. Similar statement holds for the existence of a nontrivial (indecomposable) triple and higher order Massey products in cohomology of such a manifold, see also~\cite{DS}.

Next, we are going to classify real and complex canonical moment-angle manifolds of low Lusternik-Schnirelmann category. Recall that the \emph{LS category} $\cat(X)$ of a topological space $X$ equals the minimal integer $n$ such that $X$ can be covered by $(n+1)$ open sets $U_i$ contractible to a point in $X$ for each $0\leq i\leq n$. The following result shows that~\cite[Conjecture 1.3]{LV} holds for Bier spheres. Here we use the notation $\cp(X)$ for the cohomology length of a topological space $X$.

\begin{theorem}\label{LScatBierTh}
Let $K\neq\Delta_{[m]}$ be a simplicial complex on $[m]$ with $m\geq 2$. 
\begin{itemize}
\item[(a)] The following conditions are equivalent:
\begin{itemize}
\item[(1)] $\cat(\mathcal R_{\Bier(K)})=1$; 
\item[(2)] $\cat(\mathcal Z_{\Bier(K)})=1$;
\item[(3)] $\cp(\mathcal R_{\Bier(K)})=1$; 
\item[(4)] $\cp(\mathcal Z_{\Bier(K)})=1$;
\item[(5)] $\mathcal R_{\Bier(K)}$ is homeomorphic to a sphere;
\item[(6)] $\mathcal Z_{\Bier(K)}$ is homeomorphic to a sphere;
\item[(7)] either $K=\varnothing_{[m]}$, or $K=\partial\Delta_{[m]}$.
\end{itemize}
\item[(b)] The following conditions are equivalent:
\begin{itemize}
\item[(1)] $\cat(\mathcal R_{\Bier(K)})=2$;
\item[(2)] $\cat(\mathcal Z_{\Bier(K)})=2$;
\item[(3)] $\cp(\mathcal R_{\Bier(K)})=2$;
\item[(4)] $\cp(\mathcal Z_{\Bier(K)})=2$;
\item[(5)] $\mathcal R_{\Bier(K)}$ is homeomorphic to a connected sum of sphere products with two spheres in each product;
\item[(6)] $\mathcal Z_{\Bier(K)}$ is homeomorphic to a connected sum of sphere products with two spheres in each product;
\item[(7)] either $K$, or $K^\vee$ has dimension $0$.
\end{itemize}
\end{itemize}
\end{theorem}
\begin{proof}
Statement (a) follows from the fact that a closed orientable (real) moment-angle manifold over a starshaped sphere $\mathcal K$ has cohomology length greater than or equal to $2$ due to Poincar\'e duality, unless the manifold is homeomorphic to a sphere. The latter is the case if and only if $\Bier(K)$ is a sphere if and only if either $K$, or $K^\vee$ equals $\partial\Delta_{[m]}$.

Let us prove statement (b). By the result of~\cite{BG}, condition (1) implies that $\cat(\mathcal Z_{\Bier(K)})\leq 2$ and by (a) $\cat(\mathcal Z_{\Bier(K)})\neq 1$. Hence (1)$\Rightarrow$(2). Furthermore, condition (2) implies $\cp(\mathcal Z_{\Bier(K)})\leq 2$ and by (a) $\cat(\mathcal Z_{\Bier(K)})\neq 1$. Hence (2)$\Rightarrow$(4) and similarly (1)$\Rightarrow$(3).
Next, each of (3) and (4) implies $\mathcal K:=\Bier(K)$ is either the boundary of a polygon, or is chordal and different from the boundary of a simplex. Indeed, otherwise $\dim\mathcal K\geq 2$ and it contains a proper 1-dimensional full subcomplex $\mathcal L=\mathcal K_I$ being a $p$-cycle with $p\geq 4$. By Poincar\'e duality, the pairing $\tilde{H}^{1}(\mathcal K_I)\otimes \tilde{H}^{m-4}(\mathcal K_{[m]\setminus I})\to \tilde{H}^{m-2}(\mathcal K)\cong\Z$ is non-degenerate. Furthermore, the pairing $\tilde{H}^{0}(\mathcal L_{\{1,3\}})\otimes\tilde{H}^{0}(\mathcal L_{[p]\setminus\{1,3\}})\to\tilde{H}^{1}(\mathcal L)\cong\Z$ is also non-degenerate. Therefore, we obtain $\cp(\mathcal Z_{\Bier(K)})\geq 3$, and hence $\cp(\mathcal R_{\Bier(K)})\geq \cp(\mathcal Z_{\Bier(K)})\geq 3$, due to the cup product description given in~\cite{Cai}, a contradiction. Due to~\cite{LV}, if $\Bier(K)$ is a boundary of a polygon, or is a chordal complex different from the boundary of a simplex, then the condition (7) holds. The same result implies that (7) yields $\Bier(K)$ is a stacked sphere different from the boundary of a simplex, and hence both $\mathcal R_{\Bier(K)}$ and $\mathcal Z_{\Bier(K)}$ are diffeomorphic to connected sums of sphere products with two spheres in each product due to~\cite{GLdM}. Therefore, (7)$\Rightarrow$(5) and (6). The above mentioned connected sum has LS category two, since for any closed, connected, orientable manifolds $M$ and $N$ one has $\cat(M\#N)=\max\{\cat(M),\cat(N)\}$ by the result of~\cite{DrSad}. Therefore, (5)$\Rightarrow$(1). Finally, suppose (6) holds. By the result of~\cite{Am}, it follows that $\Bier(K)$ is minimally non-Golod, hence (7) holds due to~\cite{LZ}. This finishes the proof. 
\end{proof}

In order to describe the relationship between canonical moment-angle manifolds and the topological counterparts of canonical algebraic toric manifolds, we introduce the notion of the canonical characteristic matrix. Here we refer the reader to the notation used in Section 3. 

\begin{definition}
Let $K\neq\Delta_{[m]}$ be a simplicial complex on $[m]$ with $m\geq 2$. Define the \emph{canonical characteristic matrix} $\Lambda_K$ of $K$ as an integer matrix of the size $(m-1)\times f_0(\Bier(K))$ whose columns correspond to the geometric vertices of $\Bier(K)$ by the formulae:
$$
i\mapsto -e_i\text{ for all }1\leq i\leq m\text{ such that }i\in V(\Bier(K)),\text{ and }i'\mapsto e_i\text{ for all }1\leq i\leq m\text{ such that }i'\in V(\Bier(K)).
$$    
\end{definition}

Observe that the characteristic matrix introduced in~\cite{LS} is a mod $2$ reduction $\bar{\Lambda}_K$ of $\Lambda_K$. Define $H$ and $H_{\R}$ to be the kernels of the corresponding group homomorphisms, respectively:
$$
\exp\Lambda_K\colon\mathbb T^{|V|}\to\mathbb T^{m-1}\text{ and }\exp\bar{\Lambda}_K\colon\mathbb Z_2^{|V|}\to\mathbb Z_2^{m-1},\text{ where }V=V(\Bier(K)).
$$

It follows from the definition of the canonical characteristic matrix that $H$ and $H_\R$ will be toric subgroups of the maximal possible rank that are acting freely and smoothly on the corresponding moment-angle manifolds.

\begin{definition}
Let $K\neq\Delta_{[m]}$ be a simplicial complex on $[m]$ with $m\geq 2$. Define the \emph{canonical (topological) toric manifold} over $K$ and its real counterpart as a quotient space of the canonical (real) moment-angle manifold over the simplicial complex $K$:
$$
M_K:=\mathcal Z_{\Bier(K)}/H\text{ and }M^{\R}_K:=\mathcal R_{\Bier(K)}/H_{\R}.
$$
\end{definition}

Note that the above definition of a canonical toric manifold $M_K$ agrees with the one given in the proof of Proposition~\ref{ToricNonQuasitoricProp}. Finally, we discuss the orientability of canonical real toric manifolds of simplicial complexes.

\begin{proposition}\label{OrientabilityProp}
Let $K\neq\Delta_{[m]}$ be a simplicial complex on $[m]$ with $m\geq 2$. Then its canonical real toric manifold $M_K^{\R}$ is orientable if and only if $m$ is even.
\end{proposition}
\begin{proof}
If the canonical complete regular fan $\Sigma_K$ is a normal fan of a Delzant polytope $P$, then the canonical real toric manifold $M^{\R}_K$ is a small cover $M_{\R}(P,\bar{\Lambda}_K)$. Consider the group homomorphism
$$
\epsilon\colon\Z_2^{m-1}\to\Z_2\text{ such that }e_j\mapsto 1\text{ for each }1\leq j\leq m-1.
$$
Due to~\cite[Theorem 1.7]{NN}, the small cover $M_{\R}(P,\bar{\Lambda}_K)$ is orientable if and only if the image of the composition $\epsilon\bar{\Lambda}_K$ equals $\{1\}$. This implies the statement, since $\pm e_m=\mp (e_1+\cdots+e_{m-1})\in\R^{m-1}$. If $\Sigma_K$ is not a normal fan of a Delzant polytope, a straightforward modification of the proof of~\cite[Theorem 1.7]{NN} with the starshaped region $P$ bounded by the canonical starshaped realization of $\Bier(K)$ constructed in~\cite{Zivaljevic19} shows the conclusion of the orientability criterion is still valid in our case, which finishes the proof.
\end{proof}


\normalsize

\end{document}